\documentclass[reqno,a4paper]{amsart}
\usepackage{amssymb}
\usepackage[colorlinks=true,hypertex]{hyperref}
\allowdisplaybreaks[4]
\theoremstyle{plain}
\newtheorem{thm}{Theorem}
\newtheorem{lem}{Lemma}
\theoremstyle{remark}
\newtheorem{rem}{Remark}
\DeclareMathOperator{\td}{d\mspace{-2mu}}
\date{Commenced on 21 December 2008 and completed on 24 December 2008 in Melbourne}
\date{}

\begin{document}

\title{Refinements of lower bounds for polygamma functions}

\author[F. Qi]{Feng Qi}
\address[F. Qi]{Research Institute of Mathematical Inequality Theory, Henan Polytechnic University, Jiaozuo City, Henan Province, 454010, China}
\email{\href{mailto: F. Qi <qifeng618@gmail.com>}{qifeng618@gmail.com}, \href{mailto: F. Qi <qifeng618@hotmail.com>}{qifeng618@hotmail.com}, \href{mailto: F. Qi <qifeng618@qq.com>}{qifeng618@qq.com}}
\urladdr{\url{http://qifeng618.spaces.live.com}}

\author[B.-N. Guo]{Bai-Ni Guo}
\address[B.-N. Guo]{School of Mathematics and Informatics, Henan Polytechnic University, Jiaozuo City, Henan Province, 454010, China}
\email{\href{mailto: B.-N. Guo <bai.ni.guo@gmail.com>}{bai.ni.guo@gmail.com}, \href{mailto: B.-N. Guo <bai.ni.guo@hotmail.com>}{bai.ni.guo@hotmail.com}}
\urladdr{\url{http://guobaini.spaces.live.com}}

\begin{abstract}
In the paper, some lower bounds for polygamma functions are refined.
\end{abstract}



\thanks{The first author was partially supported by the China Scholarship Council}

\thanks{This paper was typeset using \AmS-\LaTeX}

\maketitle

\section{Introduction and main results}
It is well-known that the classical Euler's gamma function
\begin{equation}\label{gamma-dfn}
\Gamma(x)=\int^\infty_0t^{x-1} e^{-t}\td t
\end{equation}
for $x>0$, the psi function $\psi(x)=\frac{\Gamma'(x)}{\Gamma(x)}$ and the polygamma functions $\psi^{(i)}(x)$ for $i\in\mathbb{N}$ are a series of important special functions and have much extensive applications in many branches such as statistics, probability, number theory, theory of $0$-$1$ matrices, graph theory, combinatorics, physics, engineering, and other mathematical sciences.
\par
In~\cite[Corollary~2]{egp}, the inequality
\begin{equation}\label{batir-alzer-ineq}
\psi'(x)e^{\psi(x)}<1
\end{equation}
for $x>0$ was deduced.
\par
In~\cite[Lemma~1.1]{batir-new-jipam} and~\cite[Lemma~1.1]{batir-new-rgmia}, the inequality~\eqref{batir-alzer-ineq} was recovered.
\par
In~\cite[Theorem~4.8]{forum-alzer}, by the aid of the inequality
\begin{equation}\label{positivity}
[\psi'(x)]^2+\psi''(x)>0
\end{equation}
for $x>0$, the inequality~\eqref{batir-alzer-ineq} was generalized as
\begin{equation}\label{alzer-forum-thm4.8}
(n-1)!\exp(-n\psi(x+1))<\vert\psi^{(n)}(x)\vert<(n-1)!\exp(-n\psi(x))
\end{equation}
for $x>0$ and $n\in\mathbb{N}$, which can be rearranged as
\begin{equation}\label{alzer-ineq-rew}
e^{-\psi(x+1)} <\sqrt[n]{\frac{\bigl\vert\psi^{(n)}(x)\bigr\vert }{(n-1)!}}\, <e^{-\psi(x)}
\end{equation}
for $x>0$ and $n\in\mathbb{N}$.
\par
In~\cite[Theorem~2.1]{batir-jmaa-06-05-065}, the left hand-side inequality in~\eqref{alzer-forum-thm4.8} was refined as
\begin{equation}\label{batir-thm2.1-ineq}
(n-1)!\exp\biggl(-n\psi\biggl(x+\frac12\biggr)\biggr) <\bigl\vert\psi^{(n)}(x)\bigr\vert  <(n-1)!\exp(-n\psi(x))
\end{equation}
which can be rearranged as
\begin{equation}\label{batir-thm2.1-ineq-rew}
e^{-\psi(x+1/2)} <\sqrt[n]{\frac{\bigl\vert\psi^{(n)}(x)\bigr\vert }{(n-1)!}}\, <e^{-\psi(x)}
\end{equation}
for $x>0$ and $n\in\mathbb{N}$.
\par
Furthermore, the function $\psi^{(n)}(x)$ was bounded in \cite[Theorem~2.2]{batir-jmaa-06-05-065} alternatively as
\begin{equation}\label{gen-2}
(n-1)!\biggl[\frac{\psi^{(k)}(x+1/2)}{(-1)^{k-1}(k-1)!}\biggr]^{n/k}
<\bigl\vert\psi^{(n)}(x)\bigr\vert  <(n-1)!\biggl[\frac{\psi^{(k)}(x)}{(-1)^{k-1}(k-1)!}\biggr]^{n/k}
\end{equation}
which can be rewritten as
\begin{equation}\label{gen-2-1}
\sqrt[k]{\frac{\bigl|\psi^{(k)}(x+1/2)\bigr|}{(k-1)!}}\, <\sqrt[n]{\frac{\bigl\vert\psi^{(n)}(x)\bigr\vert }{(n-1)!}}\, <\sqrt[k]{\frac{\bigl|\psi^{(k)}(x)\bigr|}{(k-1)!}}
\end{equation}
for $x>0$ and $1\le k\le n-1$.
\par
The main aim of this paper is to further refine the left-hand side inequalities in~\eqref{batir-thm2.1-ineq} and~\eqref{gen-2} or \eqref{batir-thm2.1-ineq-rew} and~\eqref{gen-2-1} as follows.

\begin{thm}\label{single-di-tri-thm}
For $1\le n\le2$, the inequality
\begin{equation}\label{single-di-tri-thm-ineq1}
\sqrt[n]{\frac{\bigl\vert\psi^{(n)}(x)\bigr\vert }{(n-1)!}}\, >e^{-\psi(1/\ln(1+1/x))}
\end{equation}
holds on $(0,\infty)$. For $n\in\mathbb{N}$ and $1\le k\le n-1$, the inequality
\begin{equation}\label{single-di-tri-thm-ineq2}
\sqrt[n]{\frac{\bigl\vert\psi^{(n)}(x)\bigr\vert }{(n-1)!}}\, >\sqrt[k]{\frac{\bigl\vert\psi^{(k)}(1/\ln(1+1/x))\bigr\vert}{(k-1)!}}\,
\end{equation}
is valid on $(0,\infty)$.
\end{thm}

\section{Lemmas}

In order to prove Theorem~\ref{single-di-tri-thm}, the following lemmas are needed.

\begin{lem}[\cite{subadditive-qi.tex, theta-new-proof.tex, subadditive-qi-3.tex}]\label{comp-thm-1}
For $k\in\mathbb{N}$, the inequalities
\begin{equation}\label{qi-psi-ineq-1}
\ln x-\frac1x<\psi(x)<\ln x-\frac1{2x}
\end{equation}
and
\begin{equation}\label{qi-psi-ineq}
\frac{(k-1)!}{x^k}+\frac{k!}{2x^{k+1}}< (-1)^{k+1}\psi^{(k)}(x) <\frac{(k-1)!}{x^k}+\frac{k!}{x^{k+1}}
\end{equation}
are valid on $(0,\infty)$.
\end{lem}

Recall~\cite[Chapter~XIII]{mpf-93} and~\cite[Chapter~IV]{widder} that a function $f(x)$ is said to be completely monotonic on an interval $I\subseteq\mathbb{R}$ if $f(x)$ has derivatives of all orders on $I$ and
\begin{equation}
0\le(-1)^{k}f^{(k)}(x)<\infty
\end{equation}
holds for all $k\geq0$ on $I$. Recall also~\cite{Atanassov, minus-one} that a function $f$ is said to be logarithmically completely monotonic on an interval $I\subseteq\mathbb{R}$ if it has derivatives of all orders on $I$ and its logarithm $\ln f$ satisfies
\begin{equation}\label{lcm-dfn}
0\le(-1)^k[\ln f(x)]^{(k)}<\infty
\end{equation}
for $k\in\mathbb{N}$ on $I$. In~\cite[Theorem~4]{minus-one}, it was proved that all logarithmically completely monotonic functions are also completely monotonic, but not conversely. This result was formally published when revising~\cite{compmon2}. For more information, please refer to~\cite{CBerg, Sharp-Ineq-Polygamma.tex}.

\begin{lem}\label{h(x)-com-mon-lem}
The inequality
\begin{equation}\label{e-1-t-1}
\psi'(t)<e^{1/t}-1
\end{equation}
holds on $(0,\infty)$. More generally, the function
\begin{equation}\label{h(x)-dfn-yang-fan}
e^{1/t}-\psi'(t)
\end{equation}
is completely monotonic on $(0,\infty)$.
\end{lem}

\begin{proof}
For the sake of convenience, denote the function~\eqref{h(x)-dfn-yang-fan} by $h(x)$. It is clear that
\begin{equation}\label{infty-h(t)=1}
\lim_{t\to\infty}h(t)=1.
\end{equation}
Direct calculation reveals that
\begin{equation}\label{h(t+1)-h(t)}
h(t+1)-h(t)=e^{1/(t+1)}-e^{1/t}+\psi'(t)-\psi'(t+1)=e^{1/(t+1)}-e^{1/t}+\frac1{t^2}
\end{equation}
and
\begin{multline}\label{e1t-e1(t+1)}
e^{1/t}-e^{1/(t+1)}=\int_0^1\frac1{(t+u)^2}e^{1/(t+u)}\td u \\ >\int_0^1\frac1{(t+u)^2}\biggl[1+\frac1{t+u}+\frac1{2(t+u)^2}\biggr]\td u
>\frac{6t(t+1)^3+1}{6t^3(t+1)^3}>\frac1{t^2}
\end{multline}
for $t\in(0,\infty)$. Hence, by the limit~\eqref{infty-h(t)=1} and the mathematical induction, we have
\begin{equation}\label{h(x)>1}
h(t)>h(t+1)>h(t+2)>\dotsm>h(t+k)>\lim_{k\to\infty}h(t+k)=1,
\end{equation}
which is equivalent to the inequality~\eqref{e-1-t-1}.
\par
It is obvious that the functions $e^{1/t}$ and $e^{1/(t+1)}$ are logarithmically completely monotonic on $(0,\infty)$ and $(-1,\infty)$ respectively, so they are also completely monotonic on $(0,\infty)$ and $(-1,\infty)$ respectively. This means that
$$
(-1)^k\bigl(e^{1/t}\bigr)^{(k)}>0,\quad (-1)^k\bigl[e^{1/(t+1)}\bigr]^{(k)}>0\quad \text{and}\quad (-1)^k\biggl(\frac1{t^2}\biggr)^{(k)}>0
$$
on $(0,\infty)$ for $k\ge0$. Equivalently, the signs of the functions
$$
\bigl(e^{1/t}\bigr)^{(2k)},\quad \bigl[e^{1/(t+1)}\bigr]^{(2k)}\quad \text{and}\quad \biggl(\frac1{t^2}\biggr)^{(2k)}
$$
are the same and they are opposite to
$$
\bigl(e^{1/t}\bigr)^{(2k-1)},\quad \bigl[e^{1/(t+1)}\bigr]^{(2k-1)}\quad \text{and}\quad \biggl(\frac1{t^2}\biggr)^{(2k-1)}
$$
for $k\ge0$ on $(0,\infty)$. As a result, the sign of the function
$$
\biggl[e^{1/(t+1)}-e^{1/t}+\frac1{t^2}\biggr]^{(2k)}
$$
is opposite to the sign of the function
$$
\biggl[e^{1/(t+1)}-e^{1/t}+\frac1{t^2}\biggr]^{(2k-1)}
$$
for $k\ge0$ on $(0,\infty)$. Therefore, from the inequality~\eqref{e1t-e1(t+1)}, it is obtained inductively that
$$
(-1)^k\biggl[e^{1/(t+1)}-e^{1/t}+\frac1{t^2}\biggr]^{(k)}<0
$$
on $(0,\infty)$ for $k\ge0$. Accordingly, by~\eqref{h(t+1)-h(t)}, it follows that
\begin{align*}
(-1)^k[h(t+1)-h(t)]^{(k)} & =(-1)^kh^{(k)}(t+1)-(-1)^kh^{(k)}(t) \\
&=(-1)^k\biggl[e^{1/(t+1)}-e^{1/t}+\frac1{t^2}\biggr]^{(k)}\\
&<0
\end{align*}
for $k\ge0$ on $(0,\infty)$. Thus,
\begin{multline*}
(-1)^kh^{(k)}(t)>(-1)^kh^{(k)}(t+1)>(-1)^kh^{(k)}(t+2)>\dotsm \\* >(-1)^kh^{(k)}(t+k)>\lim_{k\to\infty}\bigl[(-1)^kh^{(k)}(t+k)\bigr]=0
\end{multline*}
for $k\in\mathbb{N}$. Combining this with~\eqref{h(x)>1} shows that the function $h(t)$ defined by~\eqref{h(x)-dfn-yang-fan} is completely monotonic on $(0,\infty)$. The proof of Lemma~\ref{h(x)-com-mon-lem} is complete.
\end{proof}

\section{Proof of Theorem~\ref{single-di-tri-thm}}

Now we are in a position to prove Theorem~\ref{single-di-tri-thm}.
\par
Letting $\frac1{\ln(1+1/x)}=t$ in~\eqref{single-di-tri-thm-ineq1} and rearranging yields
\begin{equation}\label{single-di-tri-thm-ineq1-let}
 e^{n\psi(t)}\biggl\vert\psi^{(n)}\biggl(\frac1{e^{1/t}-1}\biggr)\biggr\vert>(n-1)!
\end{equation}
for $t\in(0,\infty)$.
\par
Utilizing the left-hand side inequalities in~\eqref{qi-psi-ineq-1} and~\eqref{qi-psi-ineq} gives
\begin{align*}
e^{n\psi(t)}\biggl\vert\psi^{(n)}\biggl(\frac1{e^{1/t}-1}\biggr)\biggr\vert &>e^{n(\ln t-1/t)} \biggl[(n-1)!\bigl(e^{1/t}-1\bigr)^n+\frac{n!}2\bigl(e^{1/t}-1\bigr)^{n+1}\biggr]\\
&=(n-1)!\frac{\bigl(e^{1/t}-1\bigr)^n}{e^{n(\ln(1/t)+1/t)}} \biggl[\frac{n}2\bigl(e^{1/t}-1\bigr)+1\biggr]\\
&=(n-1)!\frac{(e^{u}-1)^n}{u^ne^{nu}} \biggl[\frac{n}2(e^{u}-1)+1\biggr],
\end{align*}
where $u=\frac1t>0$. So, in order to prove~\eqref{single-di-tri-thm-ineq1-let}, it is sufficient to show
\begin{equation}\label{unenu-1}
\frac{(e^{u}-1)^n}{u^ne^{nu}} \biggl[\frac{n}2(e^{u}-1)+1\biggr]\ge1,\quad u>0,
\end{equation}
that is,
\begin{equation}\label{unenu-2}
(e^{u}-1)^n [n(e^{u}-1)+2]\ge 2u^ne^{nu},\quad u>0.
\end{equation}
Let
$$
f_n(u)=(e^{u}-1)^n [n(e^{u}-1)+2]- 2u^ne^{nu}
$$
on $(0,\infty)$. Straightforward differentiation gives
\begin{gather*}
f_1'(u)=2 e^u(e^u-1-u)>0,\\
f_2'(u)=2 e^u \bigl[3 e^{2 u}-2 e^u \bigl(u^2+u+2\bigr)+1\bigr],\\
\biggl[\frac{f_2'(u)}{2 e^u}\biggr]'=2 e^u \bigl(3 e^u-u^2-3 u-3\bigr)>0.
\end{gather*}
Hence, the derivative $f_2'(u)$ is also positive on $(0,\infty)$. Since $f_n(0)=0$ and the functions $f_1(u)$ and $f_2(u)$ are strictly increasing on $(0,\infty)$, it is obtained readily that the functions $f_1(u)$ and $f_2(u)$ are strictly positive on $(0,\infty)$. This shows that the inequalities~\eqref{unenu-1} and~\eqref{unenu-2} are valid on $(0,\infty)$ for $n=1,2$. Therefore, the inequality~\eqref{single-di-tri-thm-ineq1} is valid on $(0,\infty)$ for $n=1,2$.
\par
Letting $\frac1{\ln(1+1/x)}=t$ in~\eqref{single-di-tri-thm-ineq2} leads to
\begin{equation}\label{single-di-tri-thm-ineq2-rew}
\sqrt[n]{\frac{\bigl\vert\psi^{(n)}(1/(e^{1/t}-1))\bigr\vert }{(n-1)!}}\, >\sqrt[k]{\frac{\bigl\vert\psi^{(k)}(t)\bigr\vert}{(k-1)!}}\,
\end{equation}
for $t>0$, where $n\in\mathbb{N}$ and $1\le k\le n-1$. In~\cite[Lemma~1.2]{batir-jmaa-06-05-065}, the inequality
\begin{equation}\label{Lemma-1.2-batir-jmaa-06-05-065}
(-1)^n\psi^{(n+1)}(x)<\frac{n}{\sqrt[n]{(n-1)!}\,}\bigl[(-1)^{n-1}\psi^{(n)}(x)\bigr]^{1+1/n}
\end{equation}
for $x>0$ and $n\in\mathbb{N}$ was turned out, which can be restated more significantly as
\begin{equation}\label{gen-1-1}
\sqrt[n+1]{\frac{\bigl|\psi^{(n+1)}(x)\bigr|}{n!}}\, <\sqrt[n]{\frac{\bigl\vert\psi^{(n)}(x)\bigr\vert}{(n-1)!}},
\end{equation}
an equivalence of the right-hand side inequalities in~\eqref{gen-2} and~\eqref{gen-2-1}. Therefore, it is sufficient to show
\begin{equation}\label{single-2-rew}
\lim_{n\to\infty}\sqrt[n]{\frac{\bigl\vert\psi^{(n)}(1/(e^{1/t}-1))\bigr\vert }{(n-1)!}}\, \ge\psi'(t)
\end{equation}
for $t>0$.
\par
Making use of the double inequality~\eqref{qi-psi-ineq}, it is easy to acquire that
\begin{equation}\label{frack2bigl(e1t-1bigr)+1}
\begin{split}
\bigl(e^{1/t}-1\bigr)\sqrt[k]{\frac{k}{2}\bigl(e^{1/t}-1\bigr)+1}\,
&<\sqrt[k]{\frac{\bigl\vert\psi^{(k)}(1/(e^{1/t}-1))\bigr\vert}{(k-1)!}}\, \\ &<\bigl(e^{1/t}-1\bigr)\sqrt[k]{k\bigl(e^{1/t}-1\bigr)+1}\,
\end{split}
\end{equation}
for $t\in(0,\infty)$ and $k\in\mathbb{N}$. Hence,
\begin{equation}
\lim_{k\to\infty}\sqrt[k]{\frac{\bigl\vert\psi^{(k)}(1/(e^{1/t}-1))\bigr\vert}{(k-1)!}}\, =e^{1/t}-1.
\end{equation}
By virtue of the inequality~\eqref{e-1-t-1}, the inequality~\eqref{single-2-rew} follows, so the inequality~\eqref{single-di-tri-thm-ineq2} is proved.

\section{Remarks}

Finally, we would like to supply several remarks on Theorem~\ref{single-di-tri-thm}.

\begin{rem}
Since
$$
x<\frac1{\ln(1+1/x)}<x+\frac12
$$
and the function $\psi(x)$ and $\bigl\vert\psi^{(n)}(x)\bigr\vert$ for $n\in\mathbb{N}$ are strictly decreasing on $(0,\infty)$, the left-hand side inequalities in~\eqref{batir-thm2.1-ineq} and~\eqref{batir-thm2.1-ineq-rew} for $n=1,2$ and the left-hand side inequalities in~\eqref{gen-2} and~\eqref{gen-2-1} are refined, say nothing of the left-hand side inequality in~\eqref{alzer-forum-thm4.8} for $n=1,2$.
\end{rem}

\begin{rem}
The inequality~\eqref{single-di-tri-thm-ineq1} would be invalid if $n$ is big enough. In other words, the inequality~\eqref{single-di-tri-thm-ineq1} is valid not for all $n\in\mathbb{N}$. Otherwise, the inequality
\begin{equation}\label{invalid-ineq}
\lim_{n\to\infty}\sqrt[n]{\frac{\bigl\vert\psi^{(n)}(x)\bigr\vert }{(n-1)!}}\, =\frac1x \ge e^{-\psi(1/\ln(1+1/x))}
\end{equation}
should be valid on $(0,\infty)$. However, the reversed inequality of~\eqref{invalid-ineq} holds on $(0,\infty)$. This situation motivates us to naturally pose an open problem: What is the largest positive integer $n$ such that the inequality~\eqref{single-di-tri-thm-ineq1} holds on $(0,\infty)$?
\end{rem}

\begin{rem}
Rewriting~\eqref{batir-alzer-ineq} and~\eqref{single-di-tri-thm-ineq1} for $n=1$ leads to
\begin{equation}\label{L-double-ineq}
e^{-\psi(L(x,x+1))}<\psi'(x)<e^{-\psi(L(x,x))}
\end{equation}
for $x>0$, where
\begin{equation}
L(a,b)=
\begin{cases}
\dfrac{b-a}{\ln b-\ln a},& a\ne b\\
a,&a=b
\end{cases}
\end{equation}
stands for the logarithmic mean for positive numbers $a$ and $b$.
Since the logarithmic mean $L(a,b)$ is strictly increasing with respect to both $a>0$ and $b>0$ and the psi function $\psi(x)$ is also strictly increasing on $(0,\infty)$, the inequalities~\eqref{alzer-forum-thm4.8}, \eqref{batir-thm2.1-ineq}, \eqref{batir-thm2.1-ineq-rew}, \eqref{single-di-tri-thm-ineq1} and~\eqref{L-double-ineq} stimulate us to naturally ask the following question: What are the best scalars $p(n)\ge0$ and $q(n)>0$ such that the inequality
\begin{equation}\label{L-double-ineq-p-q}
e^{-\psi(L(x,x+q(n)))} <\sqrt[n]{\frac{\bigl\vert\psi^{(n)}(x)\bigr\vert }{(n-1)!}}\, <e^{-\psi(L(x,x+p(n)))}
\end{equation}
is valid on $(0,\infty)$?
\par
Similarly, the inequalities~\eqref{gen-2}, \eqref{gen-2-1} and~\eqref{single-di-tri-thm-ineq2} motivate us to pose the following open problem: What are the best constants $p(n,k)\ge0$ and $0<q(n,k)\le1$ such that the inequality
\begin{equation}
\sqrt[k]{\frac{\bigl|\psi^{(k)}(L(x,x+q(n,k)))\bigr|}{(k-1)!}}\, <\sqrt[n]{\frac{\bigl\vert\psi^{(n)}(x)\bigr\vert }{(n-1)!}}\, <\sqrt[k]{\frac{\bigl|\psi^{(k)}(L(x,x+p(n,k)))\bigr|}{(k-1)!}}
\end{equation}
holds on $(0,\infty)$ for $1\le k\le n-1$.
\end{rem}

\begin{rem}
Letting $\frac1{\ln(1+1/x)}=t$ in the reversed version of the inequality~\eqref{invalid-ineq} and taking the logarithm yields
\begin{equation}\label{(t)+ln-bigl}
\psi(t)+\ln\bigl(e^{1/t}-1\bigr)<0
\end{equation}
on $(0,\infty)$, which has been established in~\cite[Theorem~2.8]{batir-jmaa-06-05-065}, \cite[Theorem~2]{batir-new-jipam} and~\cite[Theorem~2]{batir-new-rgmia}. The increasing monotonicity of the function in the left-hand side of the inequality~\eqref{(t)+ln-bigl} was presented in \cite{alzer-expo-math-2006, property-psi-ii.tex, property-psi.tex} respectively. The strict concavity and some other generalizations of the function in the inequality~\eqref{(t)+ln-bigl} was discussed in~\cite{property-psi.tex} recently.
\end{rem}

\begin{rem}
The case $n=2$ and $k=1$ in~\eqref{single-di-tri-thm-ineq2} is
\begin{equation}
\psi''(x)+\bigl[\psi'(1/\ln(1+1/x))\bigr]^2<0
\end{equation}
on $(0,\infty)$. This refines the inequality
\begin{equation}\label{psi(x+frac12}
\psi''(x)+\biggl[\psi'\biggl(x+\frac12\biggr)\biggr]^2<0
\end{equation}
on $(0,\infty)$, the special case $n=2$ and $k=1$ of the inequality~\eqref{gen-2}. The inequality~\eqref{psi(x+frac12} was also refined in another direction and generalized in~\cite{AAM-Qi-09-PolyGamma.tex}.
\par
The inequality~\eqref{positivity}, a special case $n=1$ of the inequality~\eqref{Lemma-1.2-batir-jmaa-06-05-065}, has been generalized to the complete monotonicity and many other cases. For more information, please refer to~\cite{AAM-Qi-09-PolyGamma.tex, Sharp-Ineq-Polygamma.tex} and closely-related references therein.
\end{rem}

\begin{rem}
The generalized logarithmic mean $L(p;a,b)$ of order $p\in\mathbb{R}$ for positive numbers $a$ and $b$ with $a\ne b$ is defined in \cite[p.~385]{bullenmean} by
\begin{equation}\label{L(p;a,b)}
L(p;a,b)=
\begin{cases}
\biggl[\dfrac{b^{p+1}-a^{p+1}}{(p+1)(b-a)}\biggr]^{1/p},&p\ne-1,0;\\[1em]
\dfrac{b-a}{\ln b-\ln a},&p=-1;\\[0.8em]
\dfrac1e\biggl(\dfrac{b^b}{a^a}\biggr)^{1/(b-a)},&p=0.
\end{cases}
\end{equation}
It is known~\cite{ql, qx3} that $L(p;a,b)$ is strictly increasing with respect to $p\in\mathbb{R}$. See also~\cite{emv-log-convex-simple.tex, Cheung-Qi-Rev.tex} and closely-related references therein. Furthermore, we can pose the following more general open problem:
What are the best scalars $\lambda(n)$, $\mu(n)$, $p(n)$ and $q(n)$ such that the inequality
\begin{equation}\label{L-double-ineq-p-q-lambda}
e^{-\psi(L(\lambda(n);x,x+q(n)))} <\sqrt[n]{\frac{\bigl\vert\psi^{(n)}(x)\bigr\vert }{(n-1)!}}\, <e^{-\psi(L(\mu(n);x,x+p(n)))}
\end{equation}
is valid on $(0,\infty)$? What are the best constants $\lambda(n,k)$, $\mu(n,k)$, $p(n,k)$ and $q(n,k)$ such that the inequality
\begin{multline}
\sqrt[k]{\frac{\bigl|\psi^{(k)}(L(\lambda(n,k);x,x+q(n,k)))\bigr|}{(k-1)!}}\, <\sqrt[n]{\frac{\bigl\vert\psi^{(n)}(x)\bigr\vert }{(n-1)!}}\, \\* <\sqrt[k]{\frac{\bigl|\psi^{(k)}(L(\mu(n,k);x,x+p(n,k)))\bigr|}{(k-1)!}}
\end{multline}
holds on $(0,\infty)$ for $1\le k\le n-1$.
\end{rem}

\begin{rem}
At last, an alternative proof of the inequality~\eqref{single-di-tri-thm-ineq1} for $n=1$ is provided as follows. Letting $\frac1{\ln(1+1/x)}=t$ in~\eqref{single-di-tri-thm-ineq1} results in
\begin{equation}\label{single-di-tri-thm-ineq1-rew}
\sqrt[n]{\frac{\bigl\vert\psi^{(n)}(1/(e^{1/t}-1))\bigr\vert }{(n-1)!}}\, >e^{-\psi(t)}
\end{equation}
for $t>0$ and $n\in\mathbb{N}$. By the inequality
\begin{equation}\label{bernoulli-ineq-amm}
1+\frac{\alpha x}{1+(1-\alpha)x}\le (1+x)^\alpha\le1+\alpha x
\end{equation}
for $x>-1$ and $0\le\alpha\le1$, see~\cite[p.~128]{kuang-3rd} and~\cite[p.~533]{monthly-92-533}, we have
\begin{equation}
\sqrt[k]{\frac{k}{2}\bigl(e^{1/t}-1\bigr)+1}\, \ge1+\frac{e^{1/t}-1}{2+(k-1)\bigl(e^{1/t}-1\bigr)},\quad t>0.
\end{equation}
Combining this with the left-hand side inequality in~\eqref{frack2bigl(e1t-1bigr)+1} reveals that it suffice to show
\begin{equation}
1+\frac{e^{1/t}-1}{2+(k-1)\bigl(e^{1/t}-1\bigr)}>\frac{e^{-\psi(t)}}{e^{1/t}-1},\quad t>0,
\end{equation}
that is,
\begin{equation}
k<\frac1{{e^{-\psi(t)}}/({e^{1/t}-1})-1}-\frac2{e^{1/t}-1}+1,\quad t>0.
\end{equation}
By the left-hand side inequality in~\eqref{qi-psi-ineq-1}, it follows that
\begin{gather*}
\frac1{{e^{-\psi(t)}}/({e^{1/t}-1})-1}-\frac2{e^{1/t}-1} >\frac1{{e^{-(\ln t-1/t)}}/({e^{1/t}-1})-1}-\frac2{e^{1/t}-1}\\
=\frac1{e^{1/t}/t(e^{1/t}-1)-1}-\frac2{e^{1/t}-1}
=\frac{e^{2 u}-2 e^u u-1}{\bigl(e^u-1\bigr) \bigl(ue^u-e^u+1\bigr)}>0
\end{gather*}
and
$$
\lim_{u\to\infty}\frac{e^{2 u}-2 e^u u-1}{\bigl(e^u-1\bigr) \bigl(ue^u-e^u+1\bigr)}=0,
$$
where $u=\frac1t$. Hence $k\le1$. The inequality~\eqref{single-di-tri-thm-ineq1} for $n=1$ is proved.
\end{rem}

\end{document}